\documentclass[12pt]{article}

\usepackage{amssymb}
\usepackage{amsfonts}
\usepackage{amsmath}
\usepackage[usenames]{color}
\usepackage{mathrsfs}
\usepackage{amsfonts}
\usepackage{amssymb,amsmath}
\usepackage{CJK}
\usepackage{cite}
\usepackage{cases}
\usepackage{amsthm}

\def\cl{\centerline}

\def\vs{\vspace*}

\def\H{\mathcal{H}}

\def\Z{\mathbb{Z}}
\def\N{\mathbb{N}}

\def\C{\mathbb{C}}

\numberwithin{equation}{section}
\newtheorem{theo}{Theorem}[section]
\newtheorem{defi}[theo]{Definition}

\newtheorem{lemm}[theo]{Lemma}
\newtheorem{exam}[theo]{Example}
\newtheorem{prop}[theo]{Proposition}
\newtheorem{clai}{Claim}
\newtheorem{case}{Case}

\newtheorem{remark}[theo]{Remark}

\begin{document}
\begin{center}
{\large\bf  Two classes of non-weight  modules  over the\\ twisted Heisenberg-Virasoro algebra}

\end{center}
\cl{Haibo Chen, Jianzhi Han\footnote{Correspondence: J.~Han (jzhan@tongji.edu.cn)}, Yucai Su, Xiaoqing Yue}

\cl{\small  School of Mathematical Sciences, Tongji University, Shanghai
200092, China}
\cl{\small E-mails:
 rebel1025@126.com, jzhan@tongji.edu.cn, ycsu@tongji.edu.cn, xiaoqingyue@tongji.edu.cn}

\vs{8pt}
{\small
\parskip .005 truein
\baselineskip 3pt \lineskip 3pt

\noindent{{\bf Abstract:} In the present paper,
 we construct  two  classes of  non-weight modules  $\Omega(\lambda,\alpha,\beta)\otimes\mathrm{Ind}(M)$ and $\mathcal{M}\big(V,\Omega(\lambda,\alpha,\beta)\big)$ over
the twisted Heisenberg-Virasoro algebra, which are both associated with the modules $\Omega(\lambda,\alpha,\beta)$.
 We present the necessary and sufficient conditions under which   modules in these two classes  are irreducible and isomorphic, and also show that the irreducible modules in these two classes are new.
Finally, we construct
 non-weight modules  $\mathrm{Ind}_{\underline y,\lambda}(\C_{RS})$ and $\mathrm{Ind}_{\underline z,\lambda}(\C_{PQ})$  over the twisted Heisenberg-Virasoro  algebra and then apply the established results to give irreducible
conditions for $\mathrm{Ind}_{\underline y,\lambda}(\C_{RS})$ and $\mathrm{Ind}_{\underline z,\lambda}(\C_{PQ})$.
\vs{5pt}

}

}
\parskip .001 truein\baselineskip 6pt \lineskip 6pt
\section{Introduction}
The well-known {\it twisted Heisenberg-Virasoro algebra}  $\H$, initially studied by Arbarello et al. in \cite{ADK}, is the universal central
extension of the Lie algebra $\overline L$ of differential operators on a circle of order at most one:
$$\overline L:=\left\{f(t)\frac{d}{dt}+g(t)\,\Big|\, f(t),g(t)\in\C[t,t^{-1}]\right\}.$$
To be more precise,
$\H$    is
 an infinite dimensional complex Lie algebra
with  basis  $\{L_m,I_m,C_i
\mid  m\in \Z,\,i=1,2,3\}$ subject to  the following Lie brackets:
\begin{equation*}
\aligned
&[L_m,L_n]= (n-m)L_{m+n}+\delta_{m+n,0}\frac{m^{3}-m}{12}C_1,\\&
 [L_m,I_n]=n I_{m+n}+\delta_{m+n,0}(m^{2}+m)C_2,\\&
[I_m,I_n]=n\delta_{m+n,0}C_3,
\\&
 [\H,C_1]=[\H,C_2]= [\H,C_3]=0.
\endaligned
\end{equation*}
Clearly, the subspaces spanned by  $\{I_m, C_3 \mid  0\neq m\in\Z\}$   and
  by  $\{L_m, C_1 \mid m\in\Z\}$ are respectively the  Heisenberg algebra and the Virasoro algebra.
  Notice that  the center of $\H$ is spanned by  $\{C_0:=I_0,C_i\mid i=1,2,3\}$.
Moreover, the twisted Heisenberg-Virasoro algebra has a triangular decomposition:
$$\H=\H_-\oplus\mathfrak{h}\oplus\H_+,$$
where $\mathfrak{h}=\mathrm{span}_{\C} \{L_0, C_i\mid i=0,1,2,3\}$
and
$$\H_-=\mathrm{span}_{\C}\{L_{-m},I_{-m}\mid m\in\N\},\
\H_+=\mathrm{span}_{\C}\{L_{m},I_{m}\mid m\in\N\}.
$$
The twisted Heisenberg-Virasoro algebra  is one of the most important Lie algebras both in mathematics and in mathematical physics, whose structure theory has extensively studied (see, e.g.,  \cite{CL,SJ,LPZ}).
%It is known that the twisted Heisenberg-Virasoro algebra has close relations with the full-toroidal Lie algebras, and the $N=2$ Neveu-Schwarz superalgebra, which is one of the most important algebraic objects in superstring theory (see, e.g., \cite{FM,KL}).

A fundamental problem in the representation theory of  the  twisted Heisenberg-Virasoro algebra is to classify all its irreducible modules.  In fact, the  theory of  weight modules with all weight subspaces being finite dimensional (namely, Harish-Chandra-modules)
is well-developed. Irreducible weight modules over $\H$  with a nontrivial
finite dimensional weight subspace were proved to be  Harish-Chandra modules \cite{SS07}.
And irreducible Harish-Chandra $\H$-modules were classified in \cite{LZ1}, each of which was shown to be either  the highest (or lowest) weight module, or the  module of intermediate series, consistent with the well-known result for Virasoro algebra \cite{M}. While  weight modules with an infinite dimensional weight subspace were also studied (see \cite{CHS16,R}).

Non-weight modules constitute  the other important ingredients of the representation theory of $\H$, the study of which is definitely necessary and became popular in the last few years.
A large class of new non-weight irreducible $\H$-modules
 were constructed in  \cite{CG},  which includes  the highest weight modules and Whittaker modules.
 Non-weight $\H$-modules whose restriction
to the universal enveloping algebra of the degree-$0$ part (modulo center) are free of
rank $1$
were studied in  \cite{CG2} (see also \cite{HCS}).
While  by twisting the weight modules, the authors \cite{CHS16} obtained
a family of new non-weight  irreducible $\H$-modules.  However, the theory of  non-weight  $\H$-modules is far more from being well-developed.

As a continuation of \cite{CHS16}, we still study the representation theory of $\H$ in this paper.  But we shall be concerned with non-weight  $\H$-modules. To be more precisely, we  construct two classes of  new irreducible  non-weight $\H$-modules in the present paper, which are both related to the modules  $\Omega(\lambda,\alpha,\beta)$  (see \cite{CG2}). It is well known that an important way of constructing modules is to consider the linear tensor product of two modules, see for instance, \cite{CGZ,TZ,TZ1,Z} for such modules over the Virasoro algebra.  One class of non-weight modules is constructed by considering the tensor product module  and the other class  also stems from the linear tensor product of two modules  but on which different actions are given.

We briefly give a summary of the paper below.
 In Section $2$, we recall some known modules and  construct  a class of non-weight   modules $\mathcal{M}\big(V,\Omega(\lambda,\alpha,\beta)\big)$ over the twisted Heisenberg-Virasoro algebra.    Section 3 is devoted to    studying the
irreducibilities of tensor product modules  $\Omega(\lambda,\alpha,\beta)\otimes\mathrm{Ind}(M)$ and the reducibilities of  $\H$-modules $\mathcal{M}\big(V,\Omega(\lambda,\alpha,\beta)\big)$. We prove that $\mathcal{M}\big(V,\Omega(\lambda,\alpha,\beta)\big)$  is reducible if and  only if $V$  is a certain one-dimensional $\bar\H_{r,d}$-module.
In Section 4, we   give the necessary and sufficient conditions for two   irreducible  $\H$-modules $\Omega(\lambda_1,\alpha_1,\beta_1)\otimes \mathrm{Ind}(M_1)$
 and  $\Omega(\lambda_2,\alpha_2,\beta_2)\otimes \mathrm{Ind}(M_2)$   to be isomorphic. Moreover, we  determine the isomorphism classes  of the other class
 of non-weight  $\H$-modules $\mathcal{M}\big(V,\Omega(\lambda,\alpha,\beta)\big)$. In section $5$, we give some practical examples of  irreducible modules of  $\Omega(\lambda,\alpha,\beta)\otimes\mathrm{Ind}(M)$, and present the irreducibilities of  modules $\mathrm{Ind}_{\underline y,\lambda}(\C_{RS})$ and $\mathrm{Ind}_{\underline z,\lambda}(\C_{PQ})$.  At last,  we show that irreducible $\H$-modules in these two classes   are new by showing they are not isomorphic to each other or the known ones.
The main results of this paper are summarized in Theorems \ref{th1}, \ref{th3.3},   \ref{th2}, \ref{4.4},  \ref{th3} and   \ref{th4}.

Throughout this paper, we respectively denote by $\C,\C^*,\Z,\Z_+$  and $\N$   the sets of complex numbers, nonzero complex numbers, integers, nonnegative integers and positive integers,  and use $\mathcal{U}(\mathfrak{a})$ to denote the universal enveloping algebra of $\mathfrak{a}$.
 All vector spaces are assumed to be over $\C$.

\section{Some non-weight modules}\label{sec2}
%The purpose of this paper is to construct two classes of   non-weight modules over the twisted Heisenberg-Virasoro algebra. One  class of  modules are  tensor product modules by   taking tensor products of irreducible $\H$-modules $\Omega(\lambda,a,b)$  with irreducible $\H$-modules $\mathrm{Ind}(M)$ and the other class of non-weight modules includes $\Omega(\lambda,\alpha,\beta)$.

\subsection{Known modules  $\Omega(\lambda,\alpha,\beta)$ and $\mathrm{Ind}(M)$}
%To construct a class of tensor product $\H$-modules, we first  recall some   known modules and related results.

For $\lambda\in\C^*$, $\alpha,\beta\in\C$,  we recall the  non-weight module
\begin{equation}\label{Om1}\Omega(\lambda,\alpha,\beta):=\C[t],\end{equation}
with the action of  $\H$ defined, for $i=1,2,3$, $f(t)\in\C[t]$ and $m\in\Z$, by
$$L_m\cdot f(t)=\lambda^m(t-m\alpha)f(t-m),\ I_m\cdot f(t)=\lambda^m\beta f(t-m),\ C_i\cdot f(t)=0.$$
Then $\Omega(\lambda,\alpha,\beta)$ is irreducible
 if and only if $\alpha\in\C^*$ or $\beta\in\C^*$ (see \cite{CG2}). Notice that   this module reduces to  a Virasoro module if  $\beta=0$ (see \cite{LZK}).

Now let us recall a large class of irreducible modules for the twisted Heisenberg-Virasoro algebra,
which includes  the known irreducible modules such as   highest weight modules and Whittaker modules.
For any $e\in\Z_+$, denote by $\H_e$ the subalgebra
$$\mbox{$\sum\limits_{m\in\Z_+}$}(\C L_{m}\oplus\C  I_{m-e})\oplus\C C_1\oplus\C C_2\oplus\C C_3.$$
Take $M(c_0,c_1,c_2,c_3)$ to be an irreducible $\H_e$-module such that  $I_0,C_1,C_2$ and $C_3$ act on it as scalars $c_0,c_1,c_2,c_3$ respectively. For convenience, we briefly denote $M(c_0,c_1,c_2,c_3)$ by $M$ and form the induced $\H$-module
\begin{equation}\label{ind2.2}
\mathrm{Ind}(M):=\mathcal{U}(\H)\otimes_{\mathcal{U}(\H_{e})}M.
\end{equation}
The following theorem is obtained in \cite{CG}.
\begin{theo}\label{th2.1}
Let $e\in\Z_+$ and $M$ be an irreducible $\H_e$-module with $c_3=0$.
Assume there exists $k\in\Z_+$ such that
\begin{itemize}\lineskip0pt\parskip-1pt\item[{\rm (1)}]
$
\left\{\begin{array}{llll}\mbox{the\ action\ of}\   I_{k}    \ \mbox{on}\   M \ \mbox{is\ injective}\ \quad  &\mbox{if \ }k\neq 0,\\[4pt]
c_0+(n-1)c_2\neq0\quad \mathrm{for\ all}\ n\in\Z\setminus\{0\} &\mbox{if \ }k=0,
\end{array}\right.
$
\item[{\rm (2)}]   $I_nM=L_mM=0$ for  all $n>k$ and $m>k+e$.
\end{itemize}
Then
\begin{itemize}\lineskip0pt\parskip-1pt
\item[{\rm (i)}]
$\mathrm{Ind}(M)$ is an irreducible  $\H$-module$;$
\item[{\rm (ii)}]
the actions of $I_n,L_m$ on $\mathrm{Ind}(M)$ for all $n>k$ and $m>k+e$ are  locally nilpotent.
\end{itemize}
\end{theo}
\subsection{Construction of    $\mathcal{M}\big(V,\Omega(\lambda,\alpha,\beta)\big)$}
%In this subsection, we shall construct a class of   non-weight modules over the twisted Heisenberg-Virasoro algebra.

 For $d\in\{0,1\},r\in\Z_+$, %(the set of all nonnegative integers),
  denote by $\H_{r,d}$
the Lie subalgebra of $\H_{+,d}=\mathrm{span}_{\C}\{L_i,I_j\mid i\geq0,j\geq d\}$  generated by  $L_i,I_j$ for all $i>r,j>r+d$.
 Now we write    $\bar \H_{r,d}$
 the quotient algebra $\H_{+,d}/ \H_{r,d}$,
 and   $\bar L_i,\bar I_{i+d}$ the respective images of $L_i,I_{i+d}$
in $\bar \H_{r,d}$. Let $V$ be an $\bar \H_{r,d}$-module.
For any $\lambda, \alpha,\beta\in\C$, define an $\H$-action  on the vector space $\mathcal{M}\big(V,\Omega(\lambda,\alpha,\beta)\big):=V\otimes \C [t]$ as follows\vspace{-0.3cm}
 \begin{eqnarray}
 &&\label{Lm2.1}  L_m\big(v\otimes f(t)\big)=v\otimes\lambda^m(t-m\alpha)f(t-m)+\sum_{i=0}^r\Big(\frac{m^{i+1}}{(i+1)!}\bar L_i\Big)v\otimes \lambda^mf(t-m),\\&&
\label{Im2.2} I_m\big(v\otimes f(t)\big)=\sum_{i=0}^r\Big(\frac{m^{i+d}}{(i+d)!}\bar I_{i+d}\Big)v\otimes  \lambda^m\beta f(t-m),\\&&
\label{C1232.3}  C_i\big(v\otimes f(t)\big)=0\quad {\rm for}\ i\in\{1,2,3\}, m\in\Z,v\in V, f(t)\in\C[t].
 \end{eqnarray}

\begin{prop}
Let $d\in\{0,1\},r\in\Z_+$ and $V$ be    an $\bar \H_{r,d}$-module. Then
 $\mathcal{M}\big(V,\Omega(\lambda,\alpha,\beta)\big)$
is a non-weight $\H$-module under the actions given in \eqref{Lm2.1}-\eqref{C1232.3}.
\end{prop}
\begin{proof} Define a series of operators $x_m$ on  $\C[t]$ as follows:  $$x_mf(t)=\lambda^mf(t-m) \quad{\rm for}\  m\in\Z\ {\rm and}\  f(t)\in\C[t].$$ Then $I_nx_mf(t)=x_mI_nf(t)=I_{m+n}f(t)$ for any $m,n\in\Z$.
It follows from    \cite[Section 3]{LZ} that the relation $L_mL_n-L_nL_m=(n-m)L_{m+n}$ holds on
 $\mathcal{M}\big(V,\Omega(\lambda,\alpha,\beta)\big)$.
By \eqref{Lm2.1}-\eqref{C1232.3}, we have
\begin{eqnarray*}
&&(L_mI_n-I_nL_m)\big(v\otimes f(t)\big)\\
&=&\sum_{i=0}^r\Big(\frac{n^{i+d}}{(i+d)!}\bar I_{i+d}\Big)v\otimes  \big(L_mI_nf(t)\big)\\&&
+ \sum_{i=0}^r\Big(\frac{m^{i+1}}{(i+1)!}\bar L_i \Big)\sum_{i=0}^r\Big(\frac{n^{i+d}}{(i+d)!}\bar I_{i+d}\Big)v\otimes  \big(x_mI_nf(t)\big)
\end{eqnarray*}
\begin{eqnarray*}
&&-\sum_{i=0}^r\Big(\frac{n^{i+d}}{(i+d)!}\bar I_{i+d}\Big)v\otimes\big(I_nL_mf(t)\big)\\
&&-\sum_{i=0}^r\Big(\frac{n^{i+d}}{(i+d)!}\bar I_{i+d}\Big)\sum_{i=0}^r\Big(\frac{m^{i+1}}{(i+1)!}\bar L_i\Big)v\otimes  \big(I_nx_mf(t)\big)\\
&=&
n\sum_{i=0}^r\Big(\frac{n^{i+d}}{(i+d)!}\bar I_{i+d}\Big)v\otimes \big(I_{m+n}f(t)\big)\\&&
 +\sum_{i=0}^r\Big(\frac{m^{i+1}}{(i+1)!}\bar L_i\Big)\sum_{j=0}^r\Big(\frac{n^{j+d}}{(j+d)!}\bar I_{j+d}\Big)v\otimes  \big(I_{m+n}f(t)\big)\\&&
-\sum_{j=0}^r\Big(\frac{n^{j+d}}{(j+d)!}\bar I_{j+d}\Big)\sum_{i=0}^r\Big(\frac{m^{i+1}}{(i+1)!}\bar L_i\Big)v\otimes  \big(I_{m+n}f(t)\big)\\
&=&
n\sum_{i=0}^r\Big(\frac{n^{i+d}}{(i+d)!}\bar I_{i+d}\Big)v\otimes \big(I_{m+n}f(t)\big)\\
&&
+\sum_{i,j=0}^r\Big(\frac{m^{i+1}n^{j+d}}{(i+1)!(j+d)!}\big(\bar L_i\bar I_{j+d}-\bar I_{j+d}\bar L_i\big)\Big)v\otimes  \big(I_{m+n}f(t)\big)
\\ &=&
n\sum_{i=0}^r\sum_{j=\delta_{d,0}}^{i+1}\Big(\frac{m^{i+1-j}n^{j+d-1}}{(i+1-j)!(j+d-1)!}\bar I_{i+d}\Big)v\otimes  \big(I_{m+n}f(t)\big)
\\&=&
n\sum_{i=0}^r\Big( \frac{(m+n)^{i+d}}{(i+d)!}\bar I_{i+d}\Big)v\otimes  \big(I_{m+n}f(t)\big)=nI_{m+n}\big(v\otimes f(t)\big).
 \end{eqnarray*}
That is,  $L_mI_n-I_nL_m=nI_{n+m}$ holds on  $\mathcal{M}\big(V,\Omega(\lambda,\alpha,\beta)\big)$. Finally,  the relation   $I_mI_n-I_nI_m=0$  on $\mathcal{M}\big(V,\Omega(\lambda,\alpha,\beta)\big)$ is trivial.
 Thus,  the actions \eqref{Lm2.1}-\eqref{C1232.3} make
 $\mathcal{M}\big(V,\Omega(\lambda,\alpha,\beta)\big)$ into a non-weight $\H$-module.
\end{proof}
\begin{remark}\label{re2.1} Let $d\in\{0,1\},  r\in\Z_+$ and $V$ be  an irreducible $\bar\H_{r,d}$-module.
\begin{itemize}\lineskip0pt\parskip-1pt
\item[\rm(1)] $V$ must be infinite dimensional if ${\rm dim} V> 1$, since  any  irreducible finite dimensional module over the solvable Lie algebra $\bar\H_{r,d}$ is one-dimensional by Lie's Theorem.

\item[\rm(2)] Consider now    $V=\C v$  is one-dimensional. Then $\bar L_iv=\bar I_{i}v=0$
for any $i\in\N$, $\bar L_0v =\sigma v,\bar I_0v = \tau v$ for some $\sigma,\tau\in\C$. In this case $V$ is denoted by $V_{\sigma,\tau}$ and  it is clear that   $$\mathcal{M}\big(V_{\sigma,\tau},\Omega(\lambda,\alpha,\beta)\big)\cong\Omega(\lambda,\alpha-\sigma,\delta_{d,0}\beta\tau)\quad {\rm for\ any}\ \lambda\in\C^*\ {\rm and}\ \alpha,\beta\in\C,$$ where $\delta_{d,0}$ is the Kronecker delta.

\item[\rm(3)] Note that if $\bar I_{j+d} V=0$ for all $0\leq j\leq r$, then   $\mathcal{M}\big(V,\Omega(\lambda,\alpha,\beta)\big)$ reduces to a module over the  Virasoro algebra. So until further notice we always assume that $\bar I_{j+d} V\neq 0$ for some $0\leq j\leq r$. Choose such $r^\prime$ to be maximal such that $\bar I_{r^\prime+d} V\neq 0$. Then $\bar I_{r^\prime+d}$ is  a linear isomorphism by  \cite[Lemma 3.1]{CHS16} if $V$  is irreducible.
\end{itemize}
\end{remark}

 For $\lambda,\alpha,\beta\in\C$, the $\H$-module $A({\lambda,\alpha,\beta})$ of intermediate series has a $\C$-basis $\{v_{i}\mid i \in\Z\}$ with trivial central actions
 and
$$L_mv_n=\big(\lambda+n+m\alpha\big)v_{m+n},\ I_mv_n=\beta v_{m+n}\quad {\rm for}\ m,n\in\Z.$$
Note that $A({\lambda,\alpha,\beta})$ is reducible if and only if $\lambda\in\Z$, $\alpha\in\{0,1\}$ and $\beta=0$ (see \cite{LJ,KS}).
%We may take $\alpha$ to be zero if $\alpha\in\Z$, since $A({\alpha,\beta,\gamma})\cong A({\alpha+k,\beta,\gamma})$ for any $k\in\Z$.
% Define $A^\prime({0,0,0}):=A({0,0,0})/\C v_0,A^\prime({0,1,0}):=\oplus_{k\neq-1}\C v_k$
%and $A^\prime({\alpha,\beta,\gamma}):=A({\alpha,\beta,\gamma})$ otherwise. It is well-known (see \cite{KS} and \cite{R}) that $A^\prime({\alpha,\beta,\gamma})$ for $\alpha,\beta,\gamma\in\C$
%are irreducible modules.
Let $d\in\{0,1\},r\in\Z_+$ and $V$ be an $\bar\H_{r,d}$-module.
Define  the action of $\H$ on $\mathcal{M}\big(V,A(\lambda,\alpha,\beta)\big):=V\otimes A(\lambda,\alpha,\beta)$ as follows
 \begin{eqnarray*}
 &&L_m\big(u\otimes v_n\big)=\Big(n+\lambda+\alpha m+\sum_{i=0}^r\Big(\frac{m^{i+1}}{(i+1)!}\bar L_i\Big)\Big)u\otimes v_{m+n},\\&&
I_m\big(u\otimes v_n\big)=\sum_{i=0}^r\Big(\frac{\beta m^{i+d}}{(i+d)!}\bar I_{i+d}\Big)u\otimes   v_{m+n},\\&&
 C_1\big(u\otimes v_n\big)= C_2\big(u\otimes v_n\big)= C_3\big(u\otimes v_n\big)=0,
 \end{eqnarray*}
where  $m,n\in\Z$ and $u\in V$.   Then one can check that under the given actions as above,  $\mathcal{M}\big(V,A(\lambda,\alpha,\beta)\big)$ becomes a weight $\H$-module. %In what follows we always assume that $\gamma\neq 0$, since $\mathcal{M}\big(V,A(\alpha,\beta,\gamma)\big)$ reduces to a module over the Virasoro algebra when $\gamma=0.$

 \begin{defi}\label{defi-new}
 Let $V, W$ be two $\bar\H_{r,d}$-modules and $\alpha\in\C$.   Denote by $V^\alpha$  the $\bar\H_{r,d}$-module obtained from $V$  by modifying the $\bar L_0$-action as $\bar L_0-\alpha{\rm id}_{V}.$  An invertible linear map $\psi: V\rightarrow W$ is called an $\alpha$-isomorphism if $\psi(\bar L_i v)=\bar L_i\psi(v)$ and $\psi(\bar I_{i+d}v)=\alpha \bar I_{i+d}\psi(v)$ for any $i\in\Z$.
 \end{defi}

  It follows from the similar proof of  \cite[Theorem 4.1]{CHS16} that we have the following result.

%\begin{remark}\label{rem2.3}
%It is important to observe that
%$$\mathcal{M}\big(V_{\sigma,\tau},A(\alpha,\beta,\gamma)\big)\cong A(\alpha,\beta+\sigma,\delta_{d,0}\gamma\tau)\quad{\rm for\ any}\ \sigma,\tau,\alpha,\beta,\gamma\in\C,$$ where $\delta_{d,0}$ is the Kronecker delta.
%an $\bar \L_{0,1}$-module $V_{\sigma,\tau}$ reduces to a $\mathcal B_0$-module (see \cite{LLZ,LZ}) and as a result $\mathcal{M}\big(V_{\sigma,\tau},\Omega(\lambda,\alpha,\beta)\big)$ and  $\mathcal{M}\big(V_{\sigma,\tau}, A(\alpha,\beta,\gamma)\big)$ reduce to modules over the Virasoro algebra.
%\end{remark}

\begin{theo}\label{th2.3}
Let $d_i\in\{0,1\}, r_i\in\Z_+, \lambda_i, \alpha_i,0\neq\beta_i\in\C$ and $V_i$ be an  irreducible $\bar \H_{r_i,d_i}$-module for $i=1,2$.
Then   $\mathcal{M}\big(V_1,A(\lambda_1,\alpha_1,\beta_1)\big)$ and  $\mathcal{M}\big(V_2,A(\lambda_2,\alpha_2,\beta_2)\big)$  are isomorphic as $\H$-modules if and only  if
  $\lambda_1-\lambda_2\in\Z,d_1=d_2$
and $V_1^{\alpha_1}\cong V_2^{\alpha_2}$ are $\beta_1^{-1}\beta_2$-isomorphic as  $\bar \H_{{\rm max}\{r_1,r_2\},d_1}$-modules.
\end{theo}

\section{Irreducibilities}  %of  $\Omega(\lambda,\alpha,\beta)\otimes\mathrm{Ind}(M)$and \\ $\mathcal{M}\big(V,\Omega(\lambda,\alpha,\beta)\big)$}
 %In this section, we shall   characterise the irreducibility of $\Omega(\lambda,\alpha,\beta)\otimes\mathrm{Ind}(M)$, and    the reducibility of $\mathcal{M}\big(V,\Omega(\lambda,\alpha,\beta)\big)$.

%\subsection{Irreducibility of   $\Omega(\lambda,\alpha,\beta)\otimes\mathrm{Ind}(M)$}
%Now we shall construct a new class of   tensor product  modules over the twisted Heisenberg-Virasoro algebra, and   characterise the irreducibilities of these modules.

%Now we are ready to state the main result of this subsection.
\begin{theo}\label{th1}
Let $(\lambda,\alpha)\in(\C^*)^2$ or  $(\lambda,\beta)\in(\C^*)^2$. Assume   $\mathrm{Ind}(M)$ is  an  $\H$-module    defined by \eqref{ind2.2} for which $M$
satisfies the conditions in Theorem \ref{th2.1}. Then the tensor product $\Omega(\lambda,\alpha,\beta)\otimes\mathrm{Ind}(M)$ of $\H$-modules $\Omega(\lambda,\alpha,\beta)$ and $\mathrm{Ind}(M)$   is an irreducible $\H$-module.
\end{theo}
\begin{proof}
 For any $v\in \mathrm{Ind}(M)$,  there exists $K(v)\in\Z_+$ such that $I_m\cdot v=L_m\cdot v=0$ for all $m\geq K(v)$ by Theorem \ref{th2.1}. Suppose $P$ is a nonzero submodule of $\Omega(\lambda,\alpha,\beta)\otimes\mathrm{Ind}(M)$.
Choose a nonzero
\begin{equation*}\mbox{$w=\sum\limits_{i=0}^nt^i\otimes v_i\in P$  with $0\ne v_n\in \mathrm{Ind}(M)$   and $n\in\Z_+$ is minimal.}\end{equation*}
The case for $\alpha\in\C^*$ was proved  in \cite[Theorem 1]{TZ}, thus we only need to consider the case for $\beta\in\C^*$.
\begin{clai}
$n=0.$
\end{clai}

Let $K=\mathrm{max}\{K(v_i)\mid i=0,1,\ldots,n\}$. Then we have
$$\lambda^{-m}I_m\cdot w=\mbox{$\sum\limits_{i=0}^n$}\beta(t-m)^i\otimes v_i\in P\mbox{ \ for  $m\geq K$.} $$
Note that the right-hand side of the above  can also be written as
$$\mbox{$\sum\limits_{i=0}^n$}m^i w_i\in P $$  for  some $w_i\in \Omega(\lambda,\alpha,\beta)\otimes\mathrm{Ind}(M)$ (independent of the choice of $m$) with $w_n=\beta(-1)^n\otimes v_n\neq0.$  It follows from that $w_n\in P$.
 Thus,  $n$ must be zero by its minimality, proving the claim.

 To complete the proof, it  suffices to show the following claim.
\begin{clai}
$P=\Omega(\lambda,\alpha,\beta)\otimes\mathrm{Ind}(M).$
\end{clai}

By Claim 1, we have $1\otimes v_0\in P$ for some nonzero $v_0\in \mathrm{Ind}(M)$. Using
\begin{eqnarray*}
L_m\cdot (t^k\otimes v_0)&\!\!\!=\!\!\!&(\lambda^m(t-ma)(t-m)^k)\otimes v_0
\\&\!\!\!=\!\!\!&\lambda^m(t-m)^{k+1}\otimes v_0-\lambda^mm(\alpha-1)(t-m)^k\otimes v_0
\end{eqnarray*}
for $m\geq K(v_0), k\in \Z_+$ and by induction on $k$,
we deduce that  $t^k\otimes v_0\in P$ for $k\in\Z_+$,  i.e., $\Omega(\lambda,\alpha,\beta)\otimes v_0\subseteq P$.
It follows that $\Omega(\lambda,\alpha,\beta)\otimes \mathcal U(\H) v_0\subseteq P.$ Thus,
 $P=\Omega(\lambda,\alpha,\beta)\otimes\mathrm{Ind}(M)$, since the nonzero $\H$-submodule $\mathcal U(\H) v_0$ of ${\rm Ind}(M)$  generated by $v_0$ is equal to ${\rm Ind}(M)$ by the irreducibility of ${\rm Ind}(M)$.
\end{proof}

Now we describe the following two examples of the modules in Theorem \ref{th1},  which will be discussed in detail in Section \ref{se5}.
\begin{exam}\label{ex3.3}\rm\begin{itemize}\parskip-1pt\item[{\rm (i)}]
 Let $h\in\C,\,\underline d=(d_0,d_1,d_2,d_3)\in\C^4$ with $d_3=0$. Assume  $J_1$ is the left ideal of
$\mathcal{U}(\mathfrak{h}\oplus\H_+)$ generated by
$L_m,I_m,L_0-h$ and $C_i-d_i$ for $i=0,1,2,3$, $m\in \Z_{+}$. Denote
$\bar M:=\mathcal{U}(\mathfrak{h}\oplus\H_+)/ J_1$. Then $V=\mathrm{Ind}(\bar M)$  is the classical Verma module (see, e.g., \cite{B,SJS}).
By Theorem \ref{th2.1} (cf. \cite[Theorem 1]{B}),
 we obtain that if  $d_0+(n-1)d_2\neq0$ for  $n\in \Z\setminus \{0\}$, then $V$  is both an irreducible $\H$-module and a locally nilpotent
 module over $\H_+$.
From Theorem \ref{th1},  we obtain that $\Omega(\lambda,\alpha,\beta)\otimes V$  is an irreducible $\H$-module
if  $d_0+(n-1)d_2\neq0$ for $n\in \Z\setminus \{0\}$ and either $(\lambda,\alpha)\in(\C^*)^2$ or  $(\lambda,\beta)\in(\C^*)^2$. % Let $d_3=0$.

\item[{\rm (ii)}]
Let $(\lambda_1,\lambda_2,\mu_1)\in\C^3,\,\underline e=(e_0,e_1,e_2,e_3)\in\C^4$ with $e_3=0$. Assume that $J_2$ is the left ideal of
$\mathcal{U}(\H_+)$ generated by $L_1-\lambda_1,L_2-\lambda_2,L_3,L_4,\ldots,I_1-\mu_1,I_2,I_3,\ldots,C_i-e_i$ for $i=0,1,2,3$. Denote
$\tilde{M}:=\mathcal{U}(\H_+)/J_2$. Then $V=\mathrm{Ind}(\tilde{M})$ is the classical Whittaker module (see, e.g., \cite{LWZ,CG}).
 By Theorem \ref{th2.1} (cf. \cite[Example 10]{CG}),
  we obtain that if  $e_0+(n-1)e_2\neq0$ for all $n\in \Z\setminus \{0\}$ and $\mu_1\neq0$, then $V$ is both an irreducible $\H$-module and a locally nilpotent
 module over $\H_+^{(2)}=\mathrm{span}_{\C}\{L_m,I_m\mid m>2\}$. From Theorem \ref{th1}, we obtain that $\Omega(\lambda,\alpha,\beta)\otimes V$   is an irreducible $\H$-module if $e_0+(n-1)e_2\neq0$ for all $n\in \Z\setminus \{0\}$ and either $(\lambda,\alpha,\mu_1)\in(\C^*)^3$ or  $(\lambda,\beta,\mu_1)\in(\C^*)^3$.
\end{itemize}\end{exam}

%\subsection{Reducibility   of  $\mathcal{M}\big(V,\Omega(\lambda,\alpha,\beta)\big)$}
%The aim  of this subsection is to  give a  characterisation of  the reducibility of $\mathcal{M}\big(V,\Omega(\lambda,\alpha,\beta)\big)$.

Next we are going to characterise the reducibility of  $\mathcal{M}\big(V,\Omega(\lambda,\alpha,\beta)\big).$
For any $m\in\Z_+,n\in\Z$, denote
$$J_n^0=1\ \mbox{and}\ J_n^m=\prod_{j=n+1}^{n+m}(t-j)\quad \mbox{for}\ m>0.$$
Note that $\{J_n^m\mid m\in\Z_+\}$ forms a basis of $\Omega(\lambda,\alpha,\beta)$  for any $n\in\Z$.
By the action  of $\H$ on  $\Omega(\lambda,\alpha,\beta)$, it is easy to check that
\begin{equation}\label{lmj3.1}
L_mJ_n^k=\lambda^m(t-m\alpha)J_{m+n}^{k}\ \mbox{and}\ I_mJ_n^k=\lambda^m\beta J_{m+n}^{k}\quad {\rm for}\ m,n\in\Z,k\in\Z_+.
\end{equation}

Now we are ready to state the other main result of this section.

\begin{theo}\label{th3.3}
Let $\lambda\in\C^\ast$ and $V$ be  an irreducible $\bar \H_{r,d}$-module.   Then  $\mathcal{M}\big(V,\Omega(\lambda,\alpha,\beta)\big)$ is  reducible if and only if  $V\cong V_{\alpha,\delta_{d,0}\tau}$ for some $\tau\in\C$ such that $\delta_{d,0}\beta\tau=0$.
\end{theo}
\begin{proof}
Consider first that  $V$ is finite dimensional. Then  $V\cong V_{\sigma,\tau}$ for some  $\sigma,\tau\in\C$ and $\mathcal{M}\big(V_{\sigma,\tau},\Omega(\lambda,\alpha,\beta)\big)\cong \Omega(\lambda,\alpha-\sigma,\delta_{d,0}\beta\tau)$ by Remark \ref{re2.1}(2). But we know that $\Omega(\lambda,\alpha-\sigma,\delta_{d,0}\beta\tau)$ is reducible if and only if $\alpha=\sigma$ and $\delta_{d,0}\beta\tau=0$.    So in this case the statement is true.

Assume that $V$ is  infinite dimensional. To complete the proof, it suffices to show that $\mathcal{M}\big(V,\Omega(\lambda,\alpha,\beta)\big)$ is irreducible. For this, let $M$ be a nonzero submodule of  $\mathcal{M}\big(V,\Omega(\lambda,\alpha,\beta)\big)$.
Without loss of generality, we may
assume that $\lambda=1$. If $\beta=0$, then $\mathcal{M}\big(V,\Omega(\lambda,\alpha,\beta)\big)$ reduces to a module over the Virasoro algebra, which is irreducible  by \cite[Theorem 3.2]{LZ}.

Now consider the case $\beta\neq0$. Let $r^\prime$ be the nonnegative integer  as in  Remark \ref{re2.1}(3) and  $u=\sum_{m=0}^{p}v_m\otimes J_0^m$  a nonzero element in $M$ with $v_p\neq0$. It follows from \eqref{Im2.2} and the second relation in \eqref{lmj3.1} that
$$I_n(v_m\otimes J_0^m)=\sum_{i=0}^{r^\prime}\big(\frac{n^{i+d}}{(i+d)!}\bar I_{i+d}\big)v_m\otimes \beta J_n^m\quad{\rm for}\ n\in\Z, m\in\Z_+.$$
Then we can check that
\begin{eqnarray*}
I_kI_{n-k}(v_m\otimes J_0^m)&=&I_k\Big(\sum_{i=0}^{r^\prime}\big(\frac{(n-k)^{i+d}}{(i+d)!}\bar I_{i+d}\big)v_m\otimes  \beta J_{n-k}^m\Big)\\
&=&\sum_{i=0}^{r^\prime}\Big(\frac{k^{i+d}}{(i+d)!}\bar I_{i+d}\Big)\sum_{i=0}^{r^\prime}\Big(\frac{(n-k)^{i+d}}{(i+d)!}\bar I_{i+d}\Big)v_m\otimes \beta^2 J_{n}^m.
\end{eqnarray*}
Since $k$ is arbitrary, we can view $k$ as  a variable. Observe that the coefficient of $k^{2r^\prime+2d}$ in  $I_kI_{n-k}u$ is
 $$\frac{\beta^2}{((r^\prime+d)!)^2}\sum_{m=0}^{p}\bar I_{r^\prime+d}^2v_m\otimes  J_{n}^m\in M \quad {\rm for }\ n\in\Z.$$
Similarly, viewing $n$ as a variable we get
$\bar I_{r^\prime+d}^2v_p\otimes 1\in M$. Set $v=\bar I_{r^\prime+d}^2v_p$.  Now by
$L_0^n\big(v\otimes 1\big)=v\otimes t^n$ for $n\in\Z$ and the injectivity of $\bar I_{r^\prime+d}$ we see that
$ v\otimes \Omega(\lambda,\alpha,\beta)$ is nonzero subspace of $M$. It follows from
\begin{eqnarray*}&M\ni L_n(v\otimes J_{k-n}^m)=v\otimes(t-n\alpha)J_{k}^m+\sum_{i=0}^{r^\prime}\big(\frac{n^{i+1}}{(i+1)!}\bar L_i\big)v\otimes J_{k}^m\\
&{\rm and}\quad M\ni
I_n(v\otimes J_{k-n}^m)=\sum_{i=0}^{r^\prime}\big(\frac{n^{i+d}}{(i+d)!}\bar I_{i+d}\big)v\otimes  \beta J_{k}^m\end{eqnarray*} that
both $\sum_{i=0}^{r^\prime}\big(\frac{n^{i+1}}{(i+1)!}\bar L_i\big)v\otimes J_{k}^m$ and $\sum_{i=0}^{r^\prime}\big(\frac{n^{i+d}}{(i+d)!}\bar I_{i+d}\big)v\otimes  J_{k}^m$
lie in $M$ for $k,m\in\Z,n\in\Z_+$. In particular,  $\bar L_{i}v\otimes  \Omega(\lambda,\alpha,\beta)\in M$ and $\bar I_{i+d}v\otimes  \Omega(\lambda,\alpha,\beta)\in M$  for $i=0,1,\ldots,r^\prime.$ Then we must have $M=\mathcal{M}\big(V,\Omega(\lambda,\alpha,\beta)\big)$, since $V$ is irreducible. This shows the irreducibility of $\mathcal{M}\big(V,\Omega(\lambda,\alpha,\beta)\big)$,  completing the proof of Theorem \ref{th3.3}.
\end{proof}

\section{Isomorphism classes}  %of modules $\Omega(\lambda,\alpha,\beta)\otimes \mathrm{Ind}(M)$ and  $\mathcal{M}\big(V,\Omega(\lambda,\alpha,\beta)\big)$}
%In this section, we shall  determine the necessary and sufficient conditions for two irreducible modules respectively considered in Theorems \ref{th1} and \ref{th3.3}  to be isomorphic.

%\subsection{Isomorphism classes  of  $\Omega(\lambda,\alpha,\beta)\otimes\mathrm{Ind}(M)$} Now we present the main result of this subsection, which give the necessary and sufficient conditions for two irreducible modules $\Omega(\lambda_1,\alpha_1,\beta_1)\otimes \mathrm{Ind}(M_1)$ and   $\Omega(\lambda_2,\alpha_2,\beta_2)\otimes \mathrm{Ind}(M_2)$  to be  isomorphic.
\begin{theo}\label{th2}
Let $\alpha_1, \beta_1\in \C$, $(\lambda_1,\lambda_2,\alpha_2)\in(\C^*)^3$ or  $(\lambda_1,\lambda_2,\beta_2)\in(\C^*)^3$.
 Assume $\mathrm{Ind}(M_1)$ and $\mathrm{Ind}(M_2)$ are   $\H$-modules  defined by \eqref{ind2.2} for which $M_1$ and $M_2$ satisfy the conditions in Theorem $\ref{th2.1}$.  Then   $\Omega(\lambda_1,\alpha_1,\beta_1)\otimes \mathrm{Ind}(M_1)$ and   $\Omega(\lambda_2,\alpha_2,\beta_2)\otimes \mathrm{Ind}(M_2)$  are isomorphic
as $\H$-modules if and only if $(\lambda_1,\alpha_1,\beta_1)=(\lambda_2,\alpha_2,\beta_2)$ and $\mathrm{Ind}(M_1)\cong \mathrm{Ind}(M_2)$
 as $\H$-modules.
\end{theo}
\begin{proof}
 The ``if'' part is trivial. Now we prove the ``only if'' part. Let $\psi$ be an isomorphism from $\Omega(\lambda_1,\alpha_1,\beta_1)\otimes \mathrm{Ind}(M_1)$ to   $\Omega(\lambda_2,\alpha_2,\beta_2)\otimes \mathrm{Ind}(M_2)$.

Choose a nonzero element  $1\otimes v\in \Omega(\lambda_1,\alpha_1,\beta_1)\otimes \mathrm{Ind}(M_1)$. Assume
\begin{equation}\label{ass222}\mbox{$\psi(1\otimes v)=\sum\limits_{i=0}^nt^i\otimes v_i,$
 where $v_i\in \mathrm{Ind}(M_2)$ with $v_n\neq0.$}
\end{equation}
There exists  a positive  integer $K=\mathrm{max}\{K(v),K(v_i)\mid i=0,\ldots,n\}$ such that $I_m\cdot v=I_m\cdot v_i=L_m\cdot v=L_m\cdot v_i=0$
for all integers $m\geq K$ and $0\leq i\leq n$.

Now we consider the following   two cases.
\begin{case}
$\beta_2\in\C^*$.
\end{case}
For any
$m_1,m_2\geq K$, it follows from
$(\lambda_1^{-m_1}I_{m_1}-\lambda_1^{-m_2}I_{m_2})\cdot(1\otimes v)=0$
that
\begin{eqnarray}\label{lam4.1}
0&\!\!\!=\!\!\!&\big(\lambda_1^{-m_1}I_{m_1}-\lambda_1^{-m_2}I_{m_2}\big)\cdot \psi(1\otimes v)
\nonumber \\&\!\!\!=\!\!\!&
 \big(\lambda_1^{-m_1}I_{m_1}-\lambda_1^{-m_2}I_{m_2}\big)\cdot \mbox{$\sum\limits_{i=0}^n$}t^i\otimes v_i
\nonumber \\&\!\!\!=\!\!\!&
 \mbox{$\sum\limits_{i=0}^n$}\beta_2\Big(\big(\frac{\lambda_2}{\lambda_1}\big)^{m_1}(t-m_1)^i\otimes v_i -\big(\frac{\lambda_2}{\lambda_1}\big)^{m_2}(t-m_2)^i\otimes v_i\Big).
\end{eqnarray}
In particular, we have
 $$\Big(\big(\frac{\lambda_2}{\lambda_1}\big)^{m_1}-\big(\frac{\lambda_2}{\lambda_1}\big)^{m_2}\Big)(t^n\otimes v_n)=0\quad \mathrm{for\ all}\ m_1,m_2\geq K,$$
which forces $\lambda_1=\lambda_2$.
Whence  \eqref{lam4.1} can be   rewritten as
$$\mbox{$\sum\limits_{i=0}^n$}\Big((t-m_1)^i\otimes v_i -(t-m_2)^i\otimes v_i\Big)=0\quad \mathrm{for\ all}\ m_1,m_2\geq K.$$
Note from the above formula that $n=0$, since otherwise the coefficient $(-1)^n(1\otimes v_n)$ of $m_1^n$  would be zero,  yielding  a contradiction $v_n=0$.
Thus, by \eqref{ass222}  there exists  a linear bijection $\tau: \mathrm{Ind}(M_1)\rightarrow \mathrm{Ind}(M_2)$ such that
$$\psi(1\otimes v)=1\otimes \tau(v)\quad \mathrm{for\ all}\ v\in \mathrm{Ind}(M_1).$$
From  $\lambda_1=\lambda_2$ and $\psi\big(I_m\cdot(1\otimes v)\big)=I_m\cdot\psi(1\otimes v)$ for all  $m\geq K$, it is easy to get
$\beta_1\psi(1\otimes v)=\beta_2(1\otimes \tau(v))$,  which shows $\beta_1=\beta_2$.
Since $\psi\big(I_m\cdot(1\otimes v)\big)=I_m\cdot\psi(1\otimes v)$ for all $m\in\Z$,
 we have
$\psi\big(1\otimes (I_m\cdot v)\big)=1\otimes \big(I_m\cdot \tau(v)\big)$.
Clearly,
\begin{equation}\label{immmm4.3}
\tau(I_m\cdot v)=I_m\cdot \tau(v)\quad \mathrm{for\ all}\ m\in\Z,\,v\in \mathrm{Ind}(M_1).
\end{equation}
 For any
$m_1,m_2\geq K$ and $m_1\neq m_2$, by $$\psi\big((\lambda_1^{-m_1}L_{m_1}-\lambda_1^{-m_2}L_{m_2})\cdot (1\otimes v)\big)=(\lambda_1^{-m_1}L_{m_1}-\lambda_1^{-m_2}L_{m_2})\cdot \psi(1\otimes v),$$  we can deduce that
$(m_2-m_1)\alpha_1\psi(1\otimes v)=(m_2-m_1)\alpha_2(1\otimes \tau(v))$, which shows $\alpha_1=\alpha_2$.
Using $\psi\big(L_m\cdot(1\otimes v)\big)=L_m\cdot\psi(1\otimes v)$ for all $m\geq K$, we can conclude  that
$\psi(t\otimes v)=t\otimes \tau(v)$. Therefore, we get  $\psi\big((L_m\cdot1)\otimes v\big)=(L_m\cdot1)\otimes \tau(v)$ for all $m\in\Z$.
From  $$\psi\big(L_m\cdot(1\otimes v)\big)=L_m\cdot\psi(1\otimes v)\mbox{ for all $m\in\Z$,}$$  we can deduce that
$\psi\big(1\otimes (L_m\cdot v)\big)=1\otimes \big(L_m\cdot \tau(v)\big)$.
Hence,
\begin{equation} \label{eeeqqq4.4}
\tau(L_m\cdot v)=L_m\cdot \tau(v)\quad \mathrm{for\ all}\ m\in\Z,\ v\in \mathrm{Ind}(M_1).
\end{equation}
It is obvious that $\psi\big(C_i\cdot(1\otimes v)\big)=C_i\cdot\psi(1\otimes v)$ for $i=1,2,3,\, v\in \mathrm{Ind}(M_1)$, which implies $\tau(C_i\cdot v)=C_i\cdot \tau(v)$.
This together with \eqref{immmm4.3} and \eqref{eeeqqq4.4} show  that  $\psi$ is an  $\H$-module isomorphism  if $\beta_2\in\C^*$.

\begin{case}
$\alpha_2\in\C^*$.
\end{case}
 By the similar arguments as   in the proof of \cite[Theorem 2]{TZ}, we obtain that $\lambda_1=\lambda_2,\,\alpha_1=\alpha_2$ and
 there exists a linear bijection $\tau:\mathrm{Ind}(M_1)\rightarrow \mathrm{Ind}(M_2)$
such that
$\psi(1\otimes v)=1\otimes \tau(v)$ for all $v\in \mathrm{Ind}(M_1).$  At the same time,   we   get that
$\tau(L_m\cdot v)=L_m\cdot \tau(v)$ for all $m\in\Z,v\in \mathrm{Ind}(M_1).$
Since $\psi\big(I_m\cdot(1\otimes v)\big)=I_m\cdot\psi(1\otimes v)$ for all $m\geq K$, it is easy to see that  $\beta_1=\beta_2$.
Then  from $\psi\big(I_m\cdot(1\otimes v)\big)=I_m\cdot\psi(1\otimes v)$ and $\psi\big(C_i\cdot(1\otimes v)\big)=C_i\cdot\psi(1\otimes v)$ for $i=1,2,3,\, m\in\Z$,
we   conclude that   $\tau(I_m\cdot v)=I_m\cdot \tau(v)$ and $\tau(C_i\cdot v)=C_i\cdot \tau(v)$, respectively.
 Thus, $\mathrm{Ind}(M_1)\cong \mathrm{Ind}(M_2)$ as $\H$-modules for $\alpha_2\in \C^*$.

Combining the above two cases, we have the  isomorphism
criterion. This completes the proof of Theorem \ref{th2}.
\end{proof}

%\subsection{Isomorphism classes  of $\mathcal{M}\big(V,\Omega(\lambda,\alpha,\beta)\big)$}

For any $c\in\C$,  denote $\mathcal{I}_{c}$ by the  (maximal) ideal of $\C[L_0]$ generated by $L_0-c$. For an
$\H$-module $V$ and $n\in\Z$, set
$$V_n=V/\mathcal{I}_nV,\ \mathcal{W}(V)=\oplus_{n\in\Z}V_n.$$ %where $x$ is a formal variable.
Then the vector space $\mathcal{W}(V)$ carries the structure of a  weight (with respect to $L_0)$ $\H$-module under the
following given actions (see \cite{N,LZ}):
\begin{eqnarray*}
&L_m\cdot(v+\mathcal{I}_nV)=L_mv+\mathcal{I}_{m+n}V\quad{\rm and}\\&
 I_m\cdot(v+\mathcal{I}_nV)=I_mv+\mathcal{I}_{m+n}V\quad{\rm for}\ m,n\in\Z.
\end{eqnarray*}

\begin{lemm}\label{lemm4.2}
As $\H$-modules,  $\mathcal{W}\big(\Omega(\lambda,\alpha,\beta)\big)\cong A(0,1-\alpha,\beta)$.
\end{lemm}
\begin{proof}
By the definition of $\mathcal{I}_n$,  $\mathrm{dim}\big(\Omega(\lambda,\alpha,\beta)/\mathcal{I}_n\Omega(\lambda,\alpha,\beta)\big)=1$
for any $n\in\Z$.
Take $v_n=1+\mathcal{I}_n(\Omega(\lambda,\alpha,\beta))\in\Omega(\lambda,\alpha,\beta)/\mathcal{I}_n(\Omega(\lambda,\alpha,\beta))$.  Then
\begin{eqnarray*}
&&L_mv_n=\lambda^m(t- m\alpha)+\mathcal{I}_{m+n}(\Omega(\lambda,\alpha,\beta))
=\lambda^m(n+m(1-\alpha))v_{m+n},
 \\&&{\rm and}\quad I_mv_n=\lambda^m\beta+\mathcal{I}_{m+n}(\Omega(\lambda,\alpha,\beta))=\lambda^m\beta v_{m+n}\quad {\rm for\ any}\ m,n\in\Z.
\end{eqnarray*}
That is,  $L_mw_n=(n+m(1-\alpha))w_{m+n}$ and $I_mw_n=\beta w_{m+n}$, where $w_{n}=\lambda^nv_n$.  This completes the proof of this lemma.
\end{proof}

\begin{prop}\label{4.3}
We have the following isomorphism of  $\H$-modules
$$\mathcal{W}\Big(\mathcal{M}(V,\Omega(\lambda,\alpha,\beta))\Big)\cong\mathcal{M}\big(V,A(0,1-\alpha,\beta)\big).$$
\end{prop}
\begin{proof}
Using \eqref{Lm2.1}, we have
$$\mathcal{I}_n\Big(\mathcal{M}(V,\Omega(\lambda,\alpha,\beta))\Big)=\mathcal{I}_n\Big(V\otimes\Omega(\lambda,\alpha,\beta)\Big)=V\otimes\mathcal{I}_n(\Omega(\lambda,\alpha,\beta))\quad {\rm for}\ n\in\Z.$$
Then it follows from this and Lemma \ref{lemm4.2} that
\begin{eqnarray*}
&&\mathcal{W}\Big(\mathcal{M}(V,\Omega(\lambda,\alpha,\beta))\Big)=\bigoplus_{n\in\Z}\mathcal{M}(V,\Omega(\lambda,\alpha,\beta))_n
\\&=&\bigoplus_{n\in\Z}\Big(\mathcal{M}(V,\Omega(\lambda,\alpha,\beta))/\mathcal{I}_n\big(\mathcal{M}(V,\Omega(\lambda,\alpha,\beta))\big)\Big)
\\&=&\bigoplus_{n\in\Z}\Big(V\otimes\Omega(\lambda,\alpha,\beta)/V\otimes\mathcal{I}_n\big(\Omega(\lambda,\alpha,\beta)\big)\Big) \quad ({\rm by}\ \eqref{Lm2.1})
\\&\cong&V\otimes\bigoplus_{n\in\Z}\Big(\Omega(\lambda,\alpha,\beta)/\mathcal{I}_n\big(\Omega(\lambda,\alpha,\beta)\big)\Big)
=V\otimes\mathcal{W}\big(\Omega(\lambda,\alpha,\beta)\big)\\
&\cong& V\otimes A(0,1-\alpha,\beta)=\mathcal{M}\big(V,A(0,1-\alpha,\beta)\big).
\end{eqnarray*}\end{proof}
Since $\mathcal W$ is a functor from the category of $\H$-modules to itself (see \cite{N}), an $\H$-module isomorphism between $\mathcal{M}\big(V_1,\Omega(\lambda_1,\alpha_1,\beta_1)\big)$ and $ \mathcal{M}\big(V_2,\Omega(\lambda_2,\alpha_2,\beta_2)\big)$ would imply $d_1=d_2$
and $V_1^{\alpha_1}\cong V_2^{\alpha_2}$ are $\beta_1^{-1}\beta_2$-isomorphic (see Definition \ref{defi-new}) as  $\bar \H_{{\rm max}\{r_1,r_2\},d_1}$-modules by Proposition \ref{4.3} and Theorem \ref{th2.3}. So it is reasonable to include these into  sufficient conditions for  $\mathcal{M}\big(V_1,\Omega(\lambda_1,\alpha_1,\beta_1)\big)$ and $ \mathcal{M}\big(V_2,\Omega(\lambda_2,\alpha_2,\beta_2)\big)$ being isomorphic. In fact,  one  more condition $\lambda_1=\lambda_2$ will be enough, as stated in the following result.

\begin{theo}\label{4.4}
Let $d_i\in\{0,1\}, r_i\in\Z_+, \lambda_i, \beta_i\in\C^*, \alpha_i\in\C$ and $V_i$ be an  irreducible $\bar \H_{r_i,d_i}$-module such that  $\mathcal{M}\big(V_i,\Omega(\lambda_i,\alpha_i,\beta_i)\big)$ is irreducible  for $i=1,2$.
Then   $$\mathcal{M}\big(V_1,\Omega(\lambda_1,\alpha_1,\beta_1)\big)\cong \mathcal{M}\big(V_2,\Omega(\lambda_2,\alpha_2,\beta_2)\big)$$  as $\H$-modules if and only  if
  $d_1=d_2$, $\lambda_1=\lambda_2$
and $V_1^{\alpha_1}\cong V_2^{\alpha_2}$ are $\beta_1^{-1}\beta_2$-isomorphic  as  $\bar \H_{{\rm max}\{r_1,r_2\},d_1}$-modules.
\end{theo}
\begin{proof}
Let $$\phi: \mathcal{M}\big(V_1,\Omega(\lambda_1,\alpha_1,\beta_1)\big)\rightarrow \mathcal{M}\big(V_2,\Omega(\lambda_2,\alpha_2,\beta_2)\big)$$  be an  isomorphism of $\H$-modules. By the remark before this theorem we know that $d_1=d_2$ and  that the linear map $\varphi: V_1^{\alpha_1}\rightarrow V_2^{\alpha_2}$ inducing from $\phi$ is a  $\beta_1^{-1}\beta_2$-isomorphism of $\bar \H_{{\rm max}\{r_1,r_2\},d_1}$-modules.    It remains to show $\lambda_1=\lambda_2$.

Take any $0\neq w\in V_1$  and assume that $\phi(w\otimes 1)=\sum_{i}u_i\otimes t^i\in \mathcal{M}\big(V_2,\Omega(\lambda_2,\alpha_2,\beta_2)\big). $ Note on one hand that   $\phi$ induces an  $\H$-module isomorphism   $$\phi_A: \mathcal{M}\big(V_1,A(0,1-\alpha_1,\beta_1)\big)\rightarrow \mathcal{M}\big(V_2,A(0,1-\alpha_2,\beta_2)\big)$$ sending $w\otimes v_n$ to $\sum_{i}(\frac {\lambda_2}{\lambda_1})^n n^i u_i\otimes v_n$ for any $n\in\Z$  by  Lemma \ref{lemm4.2} and  Proposition \ref{4.3} ,  and on the other hand that $\phi_A(w\otimes v_n)=\varphi(w)\otimes v_n$ for any $n\in\Z$ (see \cite[Theorem 4.1]{CHS16}). Thus, $\varphi(w)=\sum_{i}(\frac {\lambda_2}{\lambda_1})^n n^i u_i,$ which implies $\lambda_1=\lambda_2$ and $u_i=0$ if $i\neq0.$

Conversely, let $\varphi: V_1^{\alpha_1}\rightarrow V_2^{\alpha_2}$ be a $\beta_1^{-1}\beta_2$-isomorphism of  $\bar \H_{{\rm max}\{r_1,r_2\},d_1}$-modules.  One can check the linear map $\phi: \mathcal{M}\big(V_1,\Omega(\lambda,\alpha_1,\beta_1)\big)\rightarrow \mathcal{M}\big(V_2,\Omega(\lambda,\alpha_2,\beta_2)\big)$ sending $v\otimes f(t)$ to $\varphi(v)\otimes f(t)$ is an isomorphism of $\H$-modules.
\end{proof}

%\begin{remark}
%\cite[Theorem 3.6]{LZ} should be stated as: Let $M, M^\prime$ be two infinite dimensional irreducible $\mathcal B_r$-modules, $\lambda,\lambda^\prime\in\C\setminus\{0\}$, $\beta,\beta^\prime\in\C$. Then $F(M,\Omega(\lambda,\beta))\cong F(M^\prime, \Omega(\lambda^\prime,\beta^\prime))$ if and only if $M^\beta\cong (M^\prime)^{\beta^\prime}$ and $\lambda=\lambda^\prime$.
%\end{remark}

\section{New irreducible modules}  %$\Omega(\lambda,\alpha,\beta)\otimes\mathrm{Ind}(M)$\\ and $\mathcal{M}\big(V,\Omega(\lambda,\alpha,\beta)\big)$}
In this section, we shall  show that any one of $\Omega(\lambda,\alpha,\beta)\otimes \mathrm{Ind}(M)$ and  $\mathcal{M}\big(V,\Omega(\lambda,\alpha,\beta)\big)$ is not isomorphic to  $\mathrm{Ind}(M)$ or
the  irreducible non-weight $\H$-modules defined in \cite{CHS16} and   that
$\Omega(\lambda,\alpha,\beta)\otimes \mathrm{Ind}(M)$  is not isomorphic to
  $\mathcal{M}\big(V,\Omega(\lambda^\prime,\alpha^\prime,\beta^\prime)\big)$.

%For any $e\in\Z_+$, denote by $\H_e$ the subalgebra
%$$\sum_{i\in\Z_+}(\C L_{i}\oplus\C  I_{i-e})\oplus\C C_1\oplus\C C_2\oplus\C C_3.$$
%Take  $M$ to be an irreducible $\H_e$-module such that  $I_0,C_1,C_2,C_3$ respectively act on $M$ as scalars $\mu,c_1,c_2,c_3,$ and form  the induced $\H$-module
%$$\mathrm{Ind}(M)=\mathcal{U}(\H)\otimes_{\mathcal{U}(\H_{e})}M\ (\mbox{ here $\mathcal U(\H)$ stands for  the enveloping algebra of $\H$}).$$
%It was proved in \cite{CG}  that $\mathrm{Ind}(M)$ is irreducible  and  the actions of $I_n,L_m$ on $\mathrm{Ind}(M)$ for all $n>k$ and $m>k+e$ are  locally nilpotent
%if the following two conditions hold:
%\begin{itemize}\lineskip0pt\parskip-1pt \item[\rm (I)] $c_3=0$;  \item[\rm(II)]  there exists $k\in\Z_+$ such that

%\begin{equation*}
%\left\{\begin{array}{llll}\mbox{the\ action\ of}\   I_{k}    \ \mbox{on}\   M \ \mbox{is\ injective}\ \quad  &\mbox{if \ }k\neq 0,\\[4pt]
%\mu+(n-1)c_2\neq0\ {\rm for}\ n\in\Z\setminus\{0\} &\mbox{if \ }k=0,
%\end{array}\right.
%\end{equation*}
%and $I_iM=L_jM=0$ for  all $i>k$ and $j>k+e$.

%\end{itemize}
%The tensor product $\Omega(\lambda,\alpha,\beta)\otimes\mathrm{Ind}(M)$ of $\H$-modules $\Omega(\lambda,\alpha,\beta)$ and $\mathrm{Ind}(M)$  was studied in  \cite{CSY}, which was shown to be irreducible if $\Omega(\lambda, \alpha,\beta)$ is irreducible,  conditions (I) and (II) hold.

 For any $l,m\in\Z$, $s\in\Z_+$,  define a sequence of  operators $T_{l,m}^{(s)}$ as follows
\begin{equation*}
T_{l,m}^{(s)}=\sum_{i=0}^s(-1)^{s-i} \binom{s}{i}I_{l-m-i}I_{m+i}.
\end{equation*}
\begin{lemm}\label{5.1}
Let $ \lambda\in\C^\ast, \alpha,\beta\in\C$ and  $M$ be an irreducible $\H_e$-module satisfying the conditions in Theorem $\ref{th2.1}$ and $V$   an irreducible $\bar \H_{r,d}$-module. Assume that $r^\prime$ is the maximal nonnegative integer such that $\bar I_{r^\prime+d}V\neq0$.  Then
\begin{itemize}\lineskip0pt\parskip-1pt
\item[\rm (i)] the action of $L_m$ for $m$ sufficiently large is not locally nilpotent on $\Omega(\lambda,\alpha,\beta)\otimes\mathrm{Ind}(M);$
\item[{\rm (ii)}] the action of $L_m$ for each $m\in\Z$ on  $\mathcal{M}\big(V,\Omega(\lambda,\alpha,\beta)\big)$ is not locally
  nilpotent$;$

\item[{\rm (iii)}] $T_{l,m}^{(2r^\prime+2d)}$ is a linear isomorphism of $\mathcal{M}\big(V,\Omega(\lambda,\alpha,\beta)\big)$ and $\widetilde{\mathcal M}(V,\gamma(t))$  for   $l,m\in\Z$ and $\gamma(t)\in\C[t,t^{-1}];$

\item[\rm (iv)]    $T_{l,m}^{(s)}$ $(s\geq1)$ is locally nilpotent on $\Omega(\lambda,\alpha,\beta)\otimes\mathrm{Ind}(M)$  whenever $m\gg 0$ and $l\gg m;$
\item[\rm (v)] $T_{l,m}^{(1)}$ acts nontrivially on $\Omega(\lambda,\alpha,\beta)\otimes\mathrm{Ind}(M)$ whenever  $m\ll 0$ and $l\ll m$.
 \end{itemize}
\end{lemm}
\begin{proof} (i) follows from the local nilpotency  of $L_m$ for $m$ sufficiently large  on $\mathrm{Ind}(M)$ by Theorem \ref{th2.1} and its  non-local nilpotency on $\Omega(\lambda,\alpha,\beta)$.
 (ii) and
  {\rm (iii)}  follows from   \cite[Lemma 3.3]{CHS16}.  (iv) follows from an easy observation  that $$T_{l,m}^{(s)}\Omega(\lambda,\alpha,\beta)=0 \quad {\rm for}\ l,m\in\Z$$ and $T_{l,m}^{(s)}$ is locally nilpotent on ${\rm Ind}(M)$ when $m\gg 0$ and $l\gg m.$ Note  when $m\ll 0$ and $l\ll m$ that  $I_{l-m}\notin \H_e$ and $I_{m}\notin\H_e$. It follows from this and   a direct computation that for  $0\neq 1\otimes 1\otimes v\in \Omega(\lambda,\alpha,\beta)\otimes {\rm Ind}(M)$ we have
\begin{eqnarray*}
T_{l,m}^{(1)}(1\otimes 1\otimes v)&=&1\otimes \big(-\lambda^m \beta I_{l-m}-\lambda^{l-m}\beta I_m-I_{l-m}I_m +\\
&&\lambda^{m+1}\beta I_{l-m-1}+\lambda^{l-m-1}\beta I_{m+1}+I_{l-m-1}I_{m+1}\big)\otimes v\neq0.
\end{eqnarray*}So $T_{l,m}^{(1)}\big(\Omega(\lambda,\alpha,\beta)\otimes\mathrm{Ind}(M)\big)\neq0$, proving (v).
\end{proof}

 Let $d\in\{0,1\}, r\in\Z_+$ and $V$  be an $\bar\H_{r,d}$-module.
For any $\gamma(t)=\sum_{i}c_it^i\in\C[t,t^{-1}]$,
define the  action of $\H$ on $V\otimes \C[t,t^{-1}]$ as follows
\begin{eqnarray*}\label{l6.1}
&&L_m\circ (v\otimes t^n)=(L_m+\sum_{i}c_iI_{m+i})(v\otimes t^n),
\\&&
I_m\circ (v\otimes t^n)=I_m(v\otimes t^n),\\
&&C_i\circ(v\otimes t^n)=0\quad {\rm for}\ m,n\in\Z, v\in V\ {\rm and}\ i=1,2,3.
 \end{eqnarray*}
Then $V\otimes \C[t,t^{-1}]$ carries the structure of an $\H$-module under the  above given actions, which is denoted by $\widetilde{\mathcal M}(V,\gamma(t))$. Note that  $\widetilde{\mathcal M}(V,\gamma(t))$ is a weight $\H$-module if and only if $\gamma(t)\in \C$ and also that
the $\H$-module $\widetilde{\mathcal M}(V,\gamma(t))$ for $\gamma(t)\in\C[t,t^{-1}]$ is  irreducible if and only if $V$ is irreducible (see \cite{CHS16}).

We are now ready to state the main result of this section.
\begin{prop}
Let $d\in\{0,1\}$, $r,e\in\Z_+,$  $\lambda\in\C^*$, $\alpha,\beta\in\C$, $M$ be an irreducible $\H_e$-module satisfying the conditions in  Theorem $\ref{th2.1}$ and  $V$  an irreducible $\bar \H_{r,d}$-module. Then
any of  $\Omega(\lambda,\alpha,\beta)\otimes \mathrm{Ind}(M)$ and $\mathcal{M}\big(V,\Omega(\lambda,\alpha,\beta)\big)$
is not isomorphic to  $\mathrm{Ind}(M^\prime)$ for any irreducible $\H_e$-module $M^\prime$ satisfying the conditions in Theorem $\ref{th2.1}$ or $\Omega(\lambda^\prime,\alpha^\prime,\beta^\prime)$ for any $\lambda^\prime\in\C^*$, $\alpha^\prime,\beta^\prime\in\C$,   or $\widetilde{\mathcal M}(W,\gamma(t))$ for any $\bar\H_{r,d}$-module $W$ and $\gamma(t)=\sum_{i}c_it^i\in\C[t,t^{-1}]$; and $\Omega(\lambda,\alpha,\beta)\otimes \mathrm{Ind}(M)$ is not isomorphic to $\mathcal{M}\big(W,\Omega(\lambda^\prime,\alpha^\prime,\beta^\prime)\big)$
\end{prop}
\begin{proof}  $\Omega(\lambda,\alpha,\beta)\otimes \mathrm{Ind}(M)\ncong \mathrm{Ind}(M^\prime)$ follows from Lemma \ref{5.1}(i) and Theorem \ref{th2.1};    $\mathcal{M}\big($ $V,$ $\Omega(\lambda,\alpha,\beta)\big)\ncong \mathrm{Ind}(M^\prime)$ follows from Lemma \ref{5.1}(ii)  and Theorem \ref{th2.1};   $\Omega(\lambda,\alpha,\beta)\otimes\mathrm{Ind}(M)\ncong\mathcal{M}\big(V,\Omega(\lambda^\prime,\alpha^\prime,\beta^\prime)\big) $ and $\Omega(\lambda,\alpha,\beta)\otimes \mathrm{Ind}(M)\ncong \widetilde{\mathcal M}(W,\gamma(t))$ follows  from Lemma \ref{5.1}(iii) and (iv); $\Omega(\lambda,\alpha,\beta)\otimes \mathrm{Ind}(M)\ncong \Omega(\lambda^\prime,\alpha^\prime,\beta^\prime)$ follows from Lemma \ref{5.1}(v) and the fact $T_{l,m}^{(1)}\Omega(\lambda,\alpha,\beta)=0$  for $l,m\in\Z$. Finally,  on the one hand, note that   the restriction of $L_0-\sum_{i}c_iI_{i}$ on ${W\otimes t^n}$ is the scalar $n$, namely, $L_0-\sum_{i}c_iI_{i}$  is semisimple on $\widetilde{\mathcal M}(W,\gamma(t));$
on the other hand,   $L_0-\sum_{i}c_iI_{i}$ has no eigenvector in $\mathcal{M}\big(V,\Omega(\lambda,\alpha,\beta)\big)$.
Thus,  $\mathcal{M}\big(V,\Omega(\lambda,\alpha,\beta)\big)\ncong\widetilde{\mathcal M}\big(W,\gamma(t)\big)$.
\end{proof}

\section{Applications of  $\Omega(\lambda,\alpha,\beta)\otimes\mathrm{Ind}(M)$}\label{se5}
%In this section, some applications of irreducible modules of $\Omega(\lambda,\alpha,\beta)\otimes\mathrm{Ind}(M)$ are given.

Inspired by \cite{MW,TZ}, we  construct  two classes of non-weight modules and then apply Theorem \ref{th1} to give
certain conditions   for these   modules being irreducible.

For   $\lambda\in\C^*$, we denote  $\H_\lambda^{(0)}=\mathrm{span}_{\C}\{L_m-\lambda^{m}L_0,I_0,I_m\mid m\geq1\}$ and $\H_\lambda^{(1)}=\mathrm{span}_{\C}\{L_m-\lambda^{m-1}L_1,I_n\mid m\geq2,n\geq1\}.$
It is easy to check that both $\H_\lambda^{(0)}$ and $\H_\lambda^{(1)}$  are Lie subalgebras of $\H$.
For a fixed    $RS=(r_1,r_2,s_0,s_1)\in\C^4$ and $PQ=(p_2,p_3,p_4,q_1,q_2)\in\C^5$, we  define  an $\H_\lambda^{(0)}$-action on $\C$ by
 \begin{equation}\label{LI5.2}
\aligned
(L_m-\lambda^{m}L_0)\cdot 1&=r_m\quad \mbox{for}\ m=1,2;
 \\
  (L_m-\lambda^{m}L_0)\cdot 1&=\lambda^{m-2}(m-1)r_2-\lambda^{m-1}(m-2)r_1\quad \mbox{for}\ m>2;
  \\
  I_m\cdot 1&=s_m \quad \mbox{for}\ m=0,1;
  \\
  I_m\cdot 1&=\lambda^{m-1}s_1 \quad \mbox{for}\ m>1
\endaligned
\end{equation} and an $\H_\lambda^{(1)}$-action on $\C$ by
 \begin{equation}\label{LII5.6}
\aligned
 (L_m-\lambda^{m-1}L_1)\cdot 1&=p_m\quad \mbox{for}\ m=2,3,4;
 \\
  (L_m-\lambda^{m-1}L_1)\cdot 1&=\lambda^{m-4}(m-3)p_4-\lambda^{m-3}(m-4)p_3\quad \mbox{for}\ m>4;
  \\
  I_m\cdot 1&=q_m \quad \mbox{for}\ m=1,2;
  \\
  I_m\cdot 1&=\lambda^{m-2}q_2 \quad \mbox{for}\ m>2.
\endaligned
\end{equation}
It is straightforward to verify that under the given actions $\C$ is an $\H_\lambda^{(0)}$-module and also an $\H_\lambda^{(1)}$-module, denoted by $\C_{RS}$ and $\C_{PQ}$, respectively.
For a  fixed $\underline y=(y_1,y_2,y_3)\in\C^3$ and $\underline z=(z_0,z_1,z_2,z_3)\in\C^4$, we form  the  modules $\mathrm{Ind}_{\underline y,\lambda}(\C_{RS})$  and $\mathrm{Ind}_{\underline z,\lambda}(\C_{PQ})$:
  \begin{eqnarray}\label{ind5.4}
  \mathrm{Ind}_{\underline y,\lambda}(\C_{RS})
  =\mathcal{U}(\H)\otimes_{\mathcal{U}(\H_\lambda^{(0)})}\C_{RS}\Big/\,\mbox{$\sum\limits_{i=1}^{3}$}(C_i-y_i)\mathcal{U}(\H)\otimes_{\mathcal{U}(\H_\lambda^{(0)})}\C_{RS},\\
  \label{ind5.1}
  \mathrm{Ind}_{\underline z,\lambda}(\C_{PQ})
  =\mathcal{U}(\H)\otimes_{\mathcal{U}(\H_\lambda^{(1)})}\C_{PQ}\Big/\,\mbox{$\sum\limits_{i=0}^{3}$}(C_i-z_i)\mathcal{U}(\H)\otimes_{\mathcal{U}(\H_\lambda^{(1)})}\C_{PQ}.
  \end{eqnarray}

    The following  lemma is the key to proving the main results of this section, which generalizes \cite[Lemma 6]{TZ}.
  \begin{lemm}\label{lemm5.1}
Let   $V$ be  a cyclic $\H$-module with a basis
  $$\Big\{I_{i-m}^{k_{i-m}}\cdots I_{i}^{k_{i}}L_{j-n}^{l_{j-n}}\cdots L_{j}^{l_{j}}\cdot v\,\Big|\,
   m,n,k_i,\ldots,k_{i-m},l_j,\ldots,l_{j-n}\in\Z_+\Big\},$$
where %$m,n\in\Z_+$  and
$0\neq v\in V$ is a  fixed vector, $i,j$ are fixed integers and $I_p\cdot v\in\C v,$ $L_q\cdot v\in\C v$  for
all integers $p>i,\,q>j$.
Then for  $(\lambda,\alpha)\in(\C^*)^2$ or  $(\lambda,\beta)\in(\C^*)^2$,  $\Omega(\lambda,\alpha,\beta)\otimes V$ is also a
cyclic $\H$-module with a generator $1\otimes v$  and  a basis
\begin{equation*}\label{ba5.2}
\mathcal B=\{I_{i-m}^{k_{i-m}}\cdots I_{i}^{k_{i}}L_{j-n}^{l_{j-n}}\cdots L_{j}^{l_{j}}L_{j+1}^{l_{j+1}}\cdot (1\otimes v)\mid
m,n,k_i,\ldots,k_{i-m},l_{j+1},l_j,\ldots,l_{j-n}\in\Z_+\}.
\end{equation*}
  \end{lemm}

\begin{proof}
 Observe from \eqref{Om1} that $\Omega(\lambda,\alpha,\beta)\otimes V$ has a basis
 \begin{eqnarray*}
 \mathcal{B^\prime}=\Big\{t^{l_{j+1}}\otimes I_{i-m}^{k_{i-m}}\cdots I_{i}^{k_{i}}L_{j-n}^{l_{j-n}}\cdots L_{j}^{l_{j}}\cdot  v\,\Big|\,
m,n,k_i,\ldots,
 k_{i-m},l_{j+1},l_j,\ldots,l_{j-n}\in\Z_+\Big\}.
 \end{eqnarray*}
% where $m,n\in\Z_+$ and $i,j$ are the fixed integers.
 Now we  define the following partial order ``$\prec$" on $\mathcal{B^\prime}$
   $$t^{l_{j+1}}\otimes I_{i-m_1}^{k_{i-m_1}}\cdots I_{i}^{k_{i}}L_{j-n_1}^{l_{j-n_1}}\cdots L_{j}^{l_{j}}\cdot v\prec t^{q_{j+1}}\otimes I_{i-m_2}^{p_{i-m_2}}\cdots I_{i}^{p_{i}}L_{j-n_2}^{q_{j-n_2}}\cdots L_{j}^{q_{j}}\cdot v$$
  if and only if
  $$
  (l_{j},\ldots,l_{j-n_1},k_{i},\ldots,k_{i-m_1},\underbrace{0,\ldots,0}_{m_2+n_2},l_{j+1})
  < (q_{j},\ldots,q_{j-n_2},p_{i},\ldots,p_{i-m_2},\underbrace{0,\ldots,0}_{m_1+n_1},q_{j+1})
$$
 in the lexicographical order, which  is defined
\begin{eqnarray*}
(a_1,...,a_\ell)<(b_1,...,b_\ell)
 \Longleftrightarrow \exists   k>0 \ \mathrm{such\ that}\ a_i=b_i\ \mathrm{for\ all}\ i<k\ \mathrm{and}\ a_k<b_k.
\end{eqnarray*}
 Each element of $\mathcal B$ can be written  as a linear combinations of elements in   $\mathcal{B^\prime}$:
\begin{eqnarray*}
&&I_{i-m}^{k_{i-m}}\cdots I_{i}^{k_{i}}L_{j-n}^{l_{j-n}}\cdots L_{j}^{l_{j}}L_{j+1}^{l_{j+1}}\cdot(1\otimes v)
\nonumber\\
& =&\!\!\lambda^{(j+1)l_{j+1}} t^{l_{j+1}}\otimes I_{i-m}^{k_{i-m}}\cdots I_{i}^{k_{i}}L_{j-n}^{l_{j-n}}\cdots L_{j}^{l_{j}}\cdot v+\mathrm{lower\ terms}\ (w.r.t\  \prec).
\end{eqnarray*}
 %Other cases  are  included in  lower terms in \eqref{il5.2}.
This shows that the transition matrix from $\mathcal B^\prime$ to $\mathcal B$ is upper triangular with diagonal entries nonzero. Thus,  $\mathcal B$ is a basis of $\Omega(\lambda,\alpha,\beta)\otimes V$ and  the lemma follows.
\end{proof}

Now we are ready to give some conditions under which  $\mathrm{Ind}_{\underline y,\lambda}(\C_{RS})$ is irreducible.
\begin{theo}\label{th3}
Let  $\lambda\in\C^*,\,\underline y=(y_1,y_2,y_3)\in\C^3,\,RS=(r_1,r_2,s_0,s_1)\in\C^4$ with $y_3=0$. Assume  $\mathrm{Ind}_{\underline y,\lambda}(\C_{RS})$
is defined as in \eqref{ind5.4}. Then
\begin{itemize}\lineskip0pt\parskip-1pt
\item[{\rm (i)}] $\mathrm{Ind}_{\underline y,\lambda}(\C_{RS})\cong \Omega(\lambda,\alpha,\beta)\otimes V,$
where $V$ is the classical Verma module described in Example {\rm\ref{ex3.3}}\,{\rm (i)} and $\alpha,\beta, h, d_i$ for $i=0,1,2,3$ are defined as
\begin{equation}\label{ab5.5}
\aligned
&\alpha=\lambda^{-2}(\lambda r_1-r_2),\ \beta=\lambda^{-1}s_1,\ h=\lambda^{-2}(r_2-2\lambda r_1),
\\&
d_0=s_0-\lambda^{-1}s_1,\ d_3=y_3=0,\  d_i=y_i\quad \mathrm{for}\ i=1,2.
\endaligned
\end{equation}
\item[{\rm (ii)}] $\mathrm{Ind}_{\underline y,\lambda}(\C_{RS})$ is irreducible if
$s_0-\lambda^{-1}s_1+(n-1)y_2\neq0\mbox{ for all}\ n\in \Z\setminus \{0\}, \ \mbox{and either }r_2\neq\lambda r_1\mbox{ or } s_1\neq0$.
\end{itemize}
%(??? in  Example {\rm\ref{ex3.3}} {\rm (i)}, the irreducibility has some condition, do the condition there coincide with condition here. There $d_3=0$, but here not???)
\end{theo}
\begin{proof}
{\rm (i)}
Let $\alpha,\beta,h,d_i\in\C$ for $i=0,1,2,3$ as in \eqref{ab5.5}. Then
\begin{eqnarray*}
&&r_1=-\lambda(\alpha+h),\
r_2=-\lambda^2(2\alpha+h),\
 s_1=\lambda \beta,\\&&
s_0=d_0+\beta,\ y_3=d_3=0,\ y_i=d_i\quad \mathrm{for}\  i=1,2.
\end{eqnarray*}
Denote  $v=1+J_1\in V$.
By Lemma \ref{lemm5.1} and the structure of  $V$, $\Omega(\lambda,\alpha,\beta)\otimes V$ is a cyclic module with a
generator $1\otimes v$ and has a basis
$$\mathcal{B}_1=\Big\{I_{-n}^{l_{-n}}\cdots I_{-1}^{l_{-1}}L_{-m}^{k_{-m}}\cdots
L_{-1}^{k_{-1}}L_{0}^{k_{0}}\cdot (1\otimes v)\,\Big|\,m,n, k_{-m},\ldots,k_{-1},k_0,l_{-n},\ldots,l_{-1}\in\Z_+
%,m,n\in\N
\Big\}.$$% where $m,n\in\N$.
 By Theorem \ref{th1} and the fact that $d_3=0$,
$\Omega(\lambda,\alpha,\beta)\otimes V$ is irreducible
if $d_0+(n-1)d_2\neq0$ for all $n\in \Z\setminus \{0\}$ and either $\alpha\in\C^*$ or $\beta\in\C^*$.

In $\Omega(\lambda,\alpha,\beta)\otimes V$,  we can compute that
 \begin{equation}\label{LI5.5}
\aligned
(L_m-\lambda^{m}L_0)\cdot(1\otimes v)&=-\lambda^m(m\alpha+h)(1\otimes v)
%\\&
=
r_m(1\otimes v)\quad\mathrm{for}\ m=1,2;
\\(L_m-\lambda^{m}L_0)\cdot(1\otimes v)&=-\lambda^m(m\alpha+h)(1\otimes v)
\\&=
\big(\lambda^{m-2}(m-1)r_2-\lambda^{m-1}(m-2)r_1\big)(1\otimes v)
\quad\mathrm{for}\ m>2;
\\I_m\cdot (1\otimes v)&=(\lambda^m\beta+\delta_{m,0}d_0)(1\otimes v)
%\\&
=s_m(1\otimes v)\quad\mathrm{for}\ m=0,1;
\\
I_m\cdot (1\otimes v)&=\lambda^{m}\beta(1\otimes v)
%\\&
=\lambda^{m-1}s_1(1\otimes v)\quad\mathrm{for}\ m>1
\endaligned
\end{equation}
 and $C_i\cdot (1\otimes v)=d_i(1\otimes v)=y_i(1\otimes v)$ for $i=1,2,3$.
Comparing \eqref{LI5.2}  with \eqref{LI5.5}, we deduce that there exists an $\H$-module homomorphism  (epimorphism)
$$\tau:\mathrm{Ind}_{\underline y,\lambda}(\C_{RS})\rightarrow \Omega(\lambda,\alpha,\beta)\otimes V,$$
which is uniquely determined by $\tau(\bar 1)=1\otimes v$ with
$$\bar1:=1\otimes 1+\mbox{$\sum\limits_{i=1}^{3}$}(C_i-y_i)\mathcal{U}(\H)\otimes_{\mathcal{U}(\H_\lambda^{(0)})}\C_{RS}\in\mathrm{Ind}_{\underline{y},\lambda}(\C_{RS}).$$
Clearly, $\mathrm{Ind}_{\underline y,\lambda}(\C_{RS})$ has a basis
 $$\mathcal{B}_2=
 \Big\{I_{-n}^{l_{-n}}\cdots I_{-1}^{l_{-1}}L_{-m}^{k_{-m}}\cdots L_{-1}^{k_{-1}} L_{0}^{k_{0}}\cdot  \bar 1\,\Big|\, m,n, k_{-m},\ldots,k_{-1},k_0,l_{-n},\ldots,l_{-1}\in
 \Z_+%,m,n\in\N
 \Big\}.$$
Since  $\tau|_{\mathcal{B}_2}:\mathcal{B}_2\rightarrow \mathcal{B}_1$ is a bijection,  $\tau:\mathrm{Ind}_{\underline y,\lambda}(\C_{RS})\rightarrow \Omega(\lambda,\alpha,\beta)\otimes V$  is an isomorphism.
Hence, {\rm (i)} holds.

{\rm (ii)} By  {\rm (i)} and Theorem \ref{th1},    $\mathrm{Ind}_{\underline y,\lambda}(\C_{RS})$ is irreducible if and only if $\Omega(\lambda,\alpha,\beta)\otimes V$
is irreducible. But Examples \ref{ex3.3}(i), $\Omega(\lambda,\alpha,\beta)\otimes V$
is irreducible if
$d_0+(n-1)d_2\neq0 \mathrm{\ for\ all}\ n\in \Z\setminus \{0\}$ and either $\alpha\neq0\ \mathrm{or}\ \beta\neq0$ by noting  $d_3=0$.
Thus by \eqref{ab5.5}  $\mathrm{Ind}_{\underline y,\lambda}(\C_{RS})$ is irreducible if
$s_0-\lambda^{-1}s_1+(n-1)y_2\neq0\ \mathrm{for\ all}\ n\in \Z\setminus \{0\}$ and either $r_2\neq\lambda r_1\ \mathrm{or}\ s_1\neq0.$
Theorem \ref{th3} is proved.
\end{proof}

While the irreducible conditions of  $\mathrm{Ind}_{\underline z,\lambda}(\C_{PQ})$ can be given as follows.
\begin{theo}\label{th4}
Let  $\lambda\in\C^*,\underline z=(z_0,z_1,z_2,z_3)\in\C^4,PQ=(p_2,p_3,p_4,q_1,q_2)\in\C^5$ with $z_3=0$. Assume   $\mathrm{Ind}_{\underline z,\lambda}(\C_{PQ})$
is defined as in \eqref{ind5.1}. Then
\begin{itemize}\lineskip0pt\parskip-1pt\item[{\rm (i)}] $\mathrm{Ind}_{\underline z,\lambda}(\C_{PQ})\cong \Omega(\lambda,\alpha,\beta)\otimes V$,  where $V$ is the classical
Whittaker module described in Example $\ref{ex3.3}$\,{\rm (ii)} and  $\alpha,\beta, \lambda_1, \lambda_2,\mu_1, e_i$ for $i=0,1,2,3$ are defined as
\begin{equation}\label{ab5.3}
\aligned
&\alpha=\lambda^{-4}(\lambda p_3-p_4),\ \beta=\lambda^{-2}q_2,\  \lambda_1=\lambda^{-3}(2p_4-3\lambda p_3),\\&
 \lambda_2=\lambda^{-2}(p_4-2\lambda p_3+\lambda^2p_2),\
 \mu_1=q_1-\lambda^{-1}q_2,\\&
  e_3=z_3=0,\   e_0=z_0-\lambda^{-2}q_2,\
  e_i=z_i \quad \mathrm{for}\ i=1,2;
\endaligned
\end{equation}
\item[{\rm (ii)}]  $\mathrm{Ind}_{\underline z,\lambda}(\C_{PQ})$ is irreducible if
$z_0+(n-1)z_2\neq0\ \mbox{for all}\ n\in \Z\setminus \{0\},\ \lambda q_1\neq q_2\ \mbox{and either }$ $p_4\neq\lambda p_3\ \mbox{or}\ q_2\neq0,$
\end{itemize}\end{theo}
\begin{proof}
{\rm (i)}
Let $\alpha,\beta,\lambda_1,\lambda_2,\mu_1\in\C$ be  as in \eqref{ab5.3}. Then we have
\begin{eqnarray*}
&&p_2=\lambda_2-\lambda\lambda_1-\lambda^2\alpha,
\
p_3=-\lambda^2(\lambda_1+2\lambda \alpha),
\
p_4=-\lambda^3(\lambda_1+3\lambda \alpha),
\\&&
 q_1=\mu_1+\lambda \beta,\ q_2=\lambda^2\beta,\ z_3=e_3=0,\ z_0=e_0+\beta,\  z_i=e_i \quad \mathrm{for}\ i=1,2.
\end{eqnarray*}
Denote   $v=1+J_2\in V.$  Clearly, $\H_+\cdot v\in\C v.$
Since $V$ has a basis
$$\Big\{I_{-n}^{l_{-n}}\cdots I_{-1}^{l_{-1}}L_{-m}^{k_{-m}}\cdots L_{-1}^{k_{-1}}L_{0}^{k_{0}}\cdot v\,\Big|\, m,n,k_{-m},\ldots,k_0,l_{-n},\ldots,l_{-1}
\in\Z_+%,m,n\in\N
\Big\},$$
using Lemma \ref{lemm5.1}, we see that
 $\Omega(\lambda,\alpha,\beta)\otimes V$ is cyclic with a generator $1\otimes v$ and has a basis
 \begin{eqnarray*}
 \mathcal{B}_1=\Big\{I_{-n}^{l_{-n}}\cdots I_{-1}^{l_{-1}}L_{-m}^{k_{-m}}\cdots L_{0}^{k_{0}}L_{1}^{k_{1}}\cdot (1\otimes v)\,\Big|\,m,n,
 k_{-m},\ldots,
 k_0, k_{1},
 l_{-n},\ldots,
 l_{-1}\in\Z_+\}.
  \end{eqnarray*}
% where $m,\,n\in\N$.
 By Theorem \ref{th1} and the fact that $e_3=0$,
$\Omega(\lambda,\alpha,\beta)\otimes V$ is irreducible if $e_0+(n-1)e_2\neq0$ for all $n\in \Z\setminus \{0\}$, $\mu_1\neq0$ and either $\alpha\in\C^*$ or $\beta\in\C^*$.

In $W$,  we  can compute that
 \begin{equation}\label{LII5.9}
\aligned
(L_m-\lambda^{m-1}L_1)\cdot(1\otimes v)&=\big(\lambda^m(1-m)\alpha-\lambda^{m-1}\lambda_1+\delta_{2-m,0}\lambda_2\big) (1\otimes v)
\\&
=p_m(1\otimes v)\quad
\mathrm{for}\ m=2,3,4;\\
(L_m-\lambda^{m-1}L_1)\cdot(1\otimes v)&=\big(\lambda^m(1-m)\alpha-\lambda^{m-1}\lambda_1\big) (1\otimes v)
\\&=\big(\lambda^{m-4}(m-3)p_4-\lambda^{m-3}(m-4)p_3\big) (1\otimes v)
\quad
\mathrm{for}\ m>4;
\\
I_m\cdot (1\otimes v)&=(\lambda^m\beta+\delta_{1-m,0}\mu_1)  (1\otimes v)
%\\&
=q_m (1\otimes v)
\quad \mathrm{for}\ m=1,2;
\\I_m\cdot (1\otimes v)&=\lambda^m\beta  (1\otimes v)
%\\&
=\lambda^{m-2}q_2 (1\otimes v)\quad\mathrm{for}\ m>2
\endaligned
 \end{equation}
 and $C_i\cdot (1\otimes v)=(e_i+\delta_{i,0}\beta)(1\otimes v)=z_i(1\otimes v)$ for $i=0,1,2,3$.
Comparing \eqref{LII5.6}  with \eqref{LII5.9},  there exists an $\H$-module homomorphism   (epimorphism)
$$\tau:\mathrm{Ind}_{\underline{z},\lambda}(\C_{PQ})\rightarrow \Omega(\lambda,\alpha,\beta)\otimes V,$$
which uniquely determined by $\tau(\bar 1)=1\otimes v$ with
$$\bar1:=1\otimes 1+\mbox{$\sum\limits_{i=0}^{3}$}(C_i-z_i)\mathcal{U}(\H)\otimes_{\mathcal{U}(\H_\lambda^{(1)})}\C_{PQ}\in\mathrm{Ind}_{\underline{z},\lambda}(\C_{PQ}).$$
It is clear that $\mathrm{Ind}_{\underline{z},\lambda}(\C_{PQ})$ has a basis
 $$\mathcal{B}_2=\Big\{I_{-n}^{l_{-n}}\cdots I_{-1}^{l_{-1}}L_{-m}^{k_{-m}}\cdots L_{0}^{k_{0}}L_{1}^{k_{1}}\cdot \bar 1\,\Big|\,m,n, k_{-m},\ldots,k_0,k_{1},l_{-n},\ldots,l_{-1}\in\Z_+\}.$$
%where $m,\,n\in\N$.
Since  $\tau\mid_{\mathcal{B}_2}:\mathcal{B}_2\rightarrow \mathcal{B}_1$ is a bijection,  $\tau:\mathrm{Ind}_{\underline{z},\lambda}(\C_{PQ})\rightarrow \Omega(\lambda,\alpha,\beta)\otimes V$  is an isomorphism.
This
completes the proof of part {\rm (i)}.

{\rm (ii)} By {\rm (i)} and Theorem \ref{th1},  $\mathrm{Ind}_{\underline{z},\lambda}(\C_{PQ})$ is irreducible if and only if $\Omega(\lambda,\alpha,\beta)\otimes V$ is irreducible.
   But by Examples \ref{ex3.3}(ii), $\Omega(\lambda,\alpha,\beta)\otimes V$
is irreducible
 if
$e_0+(n-1)e_2\neq0\ \mathrm{for\ all}\ n\in \Z\setminus \{0\},\ \mu_1\neq0$ and either $\alpha\neq0\ \mathrm{or}\ \beta\neq0$ by noting  $e_3=0$.
Thus by $\eqref{ab5.3}$
$\mathrm{Ind}_{\underline{z},\lambda}(\C_{PQ})$ is irreducible
 if $z_0+(n-1)z_2\neq0\ \mathrm{for\ all}\ n\in \Z\setminus \{0\},\ \lambda q_1\neq q_2$ and either $p_4\neq\lambda p_3\ \mathrm{or}\ q_2\neq0.$
\end{proof}

\bigskip
\noindent{\bf \Large Acknowledgments}

This work was supported by NSF grant Nos. 11431010, 11371278,  11501417, 11671247,  Innovation Program of Shanghai Municipal Education Commission.

\end{document}